\documentclass{amsart}

\usepackage{amssymb,amsmath,amsfonts,stmaryrd,mathrsfs,amscd}
\usepackage[sans]{dsfont}
\usepackage{fouridx}
\usepackage[dvips,cmtip,all]{xy}
\usepackage{setspace,paralist,tikz}
\usetikzlibrary{matrix,arrows,decorations.pathmorphing}
\usepackage[linkcolor=black]{hyperref}
\usepackage{lineno,color}

\newtheorem{theorem}{Theorem}[section]
\newtheorem{lemma}[theorem]{Lemma}

\theoremstyle{definition}
\newtheorem{definition}[theorem]{Definition}

\newtheorem{corollary}[theorem]{Corollary}

\theoremstyle{remark}
\newtheorem{remark}[theorem]{Remark}

\numberwithin{equation}{section}

\def\Spec{\operatorname{Spec}}
\def\proj{\operatorname{Proj}}

\def\heit{\operatorname{ht}}

\begin{document}

\title{Lifting finite self maps of equicharacteristic complete local rings}

\author{Mahdi Majidi-Zolbanin}
\address{Department of Mathematics, LaGuardia Community College of the City University of New York, Long Island City, NY 11101}
\email{mmajidi-zolbanin@lagcc.cuny.edu}

\author{Nikita Miasnikov}
\address{Department of Mathematics,The Graduate Center of the City University of New York, New York, NY 10016}
\email{n5k5t5@gmail.com}

\author{Lucien Szpiro}
\address{Department of Mathematics,The Graduate Center of the City University of New York, New York, NY 10016}
\email{lszpiro@gc.cuny.edu}
\thanks{The second and third authors received funding from the NSF Grants DMS-0854746 and DMS-0739346.}


\date{May 9th, 2011}

\dedicatory{}

\keywords{Complete local rings, Lifting finite self maps, Strong system of parameters}

\begin{abstract}
In the spirit of Fakhruddin~\cite[Corollary~2.2]{NF} and Szpiro-Bhatnagar \cite[Theorem~1]{AnupSzpiro}, we show that for an equicharacteristic complete local ring  \(A\), with a given embedding of Spec(\(A\)) in the prime spectrum Spec(\(R\)) of some complete regular local ring \(R\), any finite self map of \(A\) can be lifted to a finite self map of \(R\), keeping the given embedding.
\end{abstract}

\maketitle

\section{Introduction and preleminaries}\label{Introduction}
This note is inspired by a result of N.~Fakhruddin, on lifting polarized self maps of projective varieties to \emph{some} ambient projective space. In~\cite[Corollary~2.2]{NF} Fakhruddin showed that given a morphism (self map) \(\varphi:X\longrightarrow X\) of a projective variety \(X\) over an \emph{infinite} field \(K\) and an ample line bundle \(\mathcal{L}\) on \(X\) such that \(\varphi^*(\mathcal{L})\cong\mathcal{L}^{\otimes d}\) for some \(d\geq1\), there exists an embedding \(\imath\) of \(X\) in some \(\mathbb{P}^N_K\), given by an appropriate tensor power \(\mathcal{L}^{\otimes n}\) of \(\mathcal{L}\), and a morphism \(\psi:\mathbb{P}^N_K\longrightarrow\mathbb{P}^N_K\) such that \(\psi\circ\imath=\imath\circ\varphi\). In~\cite[Theorem~1]{AnupSzpiro} L.~Szpiro and A.~Bhatnagar relaxed some of Fakhruddin's hypotheses and showed that one can keep the same embedding of \(X\) given by \(\mathcal{L}\), and instead lift an appropriate power \(\varphi^r\) of the self map to the ambient projective space.\par In this paper we will consider the analogous lifting problem for \emph{finite} self maps of equicharacteristic complete Noetherian local rings. In Theorem~\ref{lift} we show that for an equicharacteristic complete local ring  \(A\), with a given embedding of Spec(\(A\)) in the prime spectrum Spec(\(R\)) of some complete regular local ring \(R\), any finite self map of \(A\) can be lifted to a finite self map of \(R\), keeping the given embedding. Then in Corollary~\ref{shift} we conclude that for a given equicharacteristic complete Noetherian local ring \(A\), there exists an embedding of \(\Spec(A)\) in the prime spectrum \(\Spec(R)\) of some power series ring \(R\) over a field, such that any finite self map of \(A\) can be lifted to a finite self map of \(R\). As another improvement, we do not assume our fields to be infinite. \par
In the remaining of this section we will present some preparatory facts that will be used in the proof of Theorem~\ref{lift}. 
\begin{definition}[{\cite[p.~159]{Hans}}]
In a Noetherian local ring \((R,\mathfrak{m})\) of dimension \(d\) and of embedding dimension \(e\), by a \emph{strong system of parameters} we mean a system of parameters \(\{x_1,\ldots,x_d\}\) that is part of a minimal set of generators  \(\{x_1,\ldots,x_d,\ldots,x_e\}\) of the maximal ideal \(\mathfrak{m}\).
\end{definition}
\begin{lemma}\label{SOP}
Every Noetherian local ring \((R,\mathfrak{m})\) has strong systems of parameters.
\end{lemma}
\begin{proof}
 Let \(k=R/\mathfrak{m}\), let \(e\) be the embedding dimension of \(R\), and let \(d=\dim R\). If \(d=0\) then the statement holds trivially, since every system of parameters is empty. Assume now, that \(d>0\). We will use the Prime Avoidance Lemma~\cite[{p.~2, (1.B)}]{Matsumura1} to construct a strong system of parameters, inductively. It suffices to construct a sequence of elements \(x_1,\ldots,x_d\in\mathfrak{m}\) such that
\begin{enumerate}
\item[(a)] \(\dim R/\left<x_1,\ldots,x_i\right>=d-i\), for \(1\leq i\leq d\), and
\item[(b)] the images of \(x_1,\ldots,x_d\) in \(\mathfrak{m}/\mathfrak{m}^2\) are linearly independent over \(k\).
\end{enumerate}
To choose \(x_1\), let \(\{\mathfrak{p}_{1i}\}_{1\leq i\leq t}\) be the set of minimal prime ideals of \(R\) with the property \(\dim R/\mathfrak{p}_{1i}=d\). By the Avoidance Lemma there is an element 
\[x_1\in\mathfrak{m}\setminus\left(\mathfrak{m}^2\cup\mathfrak{p}_{11}\cup\ldots\cup\mathfrak{p}_{1t}\right).\]
Then \(\dim R/\left<x_1\right>=d-1\) and the image of \(x_1\) in \(\mathfrak{m}/\mathfrak{m}^2\) is linearly independent over \(k\). Now let \(r-1<d\) and suppose we have constructed a sequence of elements \(x_1,\ldots,x_{r-1}\) in \(\mathfrak{m}\) with desired properties (a) and (b). To choose the next element \(x_{r}\), let \(\{\mathfrak{p}_{ri}\}_{1\leq i\leq s}\) be the set of all minimal associated prime ideals of \(R/\left<x_1,\ldots,x_{r-1}\right>\) that satisfy \(\dim R/\mathfrak{p}_{ri}=d-r+1\). Since \(r-1<d\leq e\), we cannot have \(\mathfrak{m}=\mathfrak{m}^2+\left<x_1,\ldots,x_{r-1}\right>\). Hence, by the Avoidance Lemma there is an element 
\[x_r\in\mathfrak{m}\setminus\left(\mathfrak{m}^2+\left<x_1,\ldots,x_{r-1}\right>\cup\mathfrak{p}_{r1}\cup\ldots\cup\mathfrak{p}_{rs}\right).\]
Then \(\dim R/\left<x_1,\ldots,x_r\right>=d-r\). To complete the proof we need to show that the images \(\overline{x}_1,\ldots,\overline{x}_r\) of \(x_1,\ldots,x_r\) in \(\mathfrak{m}/\mathfrak{m}^2\) are linearly independent over \(k\). If not, then since by the induction hypothesis \(\overline{x}_1,\ldots,\overline{x}_{r-1}\) are linearly independent over \(k\), we must have a dependence relation of the form \(\alpha_1\overline{x}_1+\ldots+\alpha_{r-1}\overline{x}_{r-1}-x_r=0\) in \(\mathfrak{m}/\mathfrak{m}^2\), with \(\alpha_i\in k\). This means if for \(1\leq i\leq r-1\) we choose elements \(a_i\in R\) such that they map to \(\alpha_i\) in \(R/\mathfrak{m}\), then \(a_1x_1+\ldots+a_{r-1}x_{r-1}-x_r\in\mathfrak{m}^2\), or \(x_r\in\mathfrak{m}^2+\left<x_1,\ldots,x_{r-1}\right>\), which contradicts the choice of \(x_r\). Thus, the images of \(x_1,\ldots,x_r\) in \(\mathfrak{m}/\mathfrak{m}^2\) must be linearly independent over \(k\).
\end{proof}
\begin{lemma}\label{lifting fields}
Let \((R,\mathfrak{m})\) be an equicharacteristic complete local ring and suppose \(A=R/\mathfrak{a}\) is a quotient ring with canonical surjection \(\pi:R\longrightarrow A\). If \(K\) is a subfield of \(A\), then there is a subfield \(L\) of \(R\) such that \(\pi|_L:L\longrightarrow K\) is an isomorphism.
\end{lemma}
\begin{proof}
Let \(B=\pi^{-1}(K)\). Then \(B\) is a local subring of \(R\) with maximal ideal \(\mathfrak{n}=\pi^{-1}(0)\). Note that as subsets of \(R\), \(\mathfrak{n}=\mathfrak{a}\). Since \(B/\mathfrak{n}= K\), we see that \(B\) is equicharacteristic, as well. In general \(B\) need not be Noetherian. We note that \(B\) is a closed subset of \(R\) and hence, is complete with respect to the topology induced from the \(\mathfrak{m}\)-adic topology of \(R\). To see this, note that the topology induced from the \(\mathfrak{m}_A\)-adic topology of \(A\) on any subfield of \(A\) is the discrete topology. Therefore, any subfield of \(A\) is complete with respect to the topology induced from \(A\), and hence is closed in \(A\). Since \(\pi:R\longrightarrow A\) is a continuous map and \(B=\pi^{-1}(K)\), the result follows. Now let \(\widehat{B}\) be the \(\mathfrak{n}\)-adic completion of \(B\). Since \(B\) is a local subring of \(R\) and \(R\) is complete, we get a map \(\widehat{i}:\widehat{B}\longrightarrow R\), where \(i:B\hookrightarrow R\) is the inclusion homomorphism. Furthermore, since \(B\) is a closed subset of \(R\) and is complete with respect to the topology induced from the \(\mathfrak{m}\)-adic topology of \(R\), we see that \(\widehat{i}(\widehat{B})=B\). Let \(L^\prime\) be a coefficient field of \(\widehat{B}\). (For the existence of coefficient fields in complete local rings that are not necessarily Noetherian, see~\cite[Theorem~31.1, p.~106]{Nagata}, or~\cite[Theorem~28.3, p.~215]{Matsumura2} or~\cite[Corollary~2, p.~47]{Geddes}.) Let \(L:=\widehat{i}(L^\prime)\). Then \(L\) is subfield of \(B\) that is isomorphic to \(L^\prime\). Furthermore, the following diagram is commutative, and shows that \(\pi|_L:L\longrightarrow B/\mathfrak{n}=K\) is an isomorphism.
\[\begin{tikzpicture}[>=angle 90]
\matrix(m)[matrix of math nodes,
row sep=2.6em, column sep=2.8em,
text height=1.5ex, text depth=0.25ex]
{L^\prime&\widehat{B}&B&L\\
&\widehat{B}/\widehat{\mathfrak{n}}&B/\mathfrak{n}&\\};
\path[->>,font=\scriptsize](m-1-2) edge node[above] {$\widehat{i}$} (m-1-3);
\path[right hook->](m-1-1) edge (m-1-2);
\path[left hook->](m-1-4) edge (m-1-3);
\path[->,font=\scriptsize](m-1-1) edge node[above,sloped] {$\simeq$} (m-2-2);
\path[->](m-1-4) edge (m-2-3);
\path[->,font=\scriptsize](m-2-2) edge node[above] {$\simeq$} (m-2-3);
\path[->>,font=\scriptsize](m-1-2) edge (m-2-2);
\path[->>,font=\scriptsize](m-1-3) edge node[left] {$\pi|_B$} (m-2-3);
\path[->,font=\scriptsize](m-1-1) edge [bend left]  node[above] {$\simeq$} (m-1-4);
\end{tikzpicture}\]

\end{proof}
\begin{lemma}\label{nice}
Let \(\varphi:(A,\mathfrak{m})\longrightarrow(B,\mathfrak{n})\) be a local homomorphism of Noetherian local rings and suppose that \(\left<\varphi(\mathfrak{m})\right>\) is \(\mathfrak{n}\)-primary. If we have a prime ideal \(P\) of \(B\) such that \(\varphi^{-1}(P)=\mathfrak{m}\) then \(P=\mathfrak{n}\).
\end{lemma}
\begin{proof}
Since \(\varphi^{-1}(P)=\mathfrak{m}\) we must have \(\left<\varphi\left(\varphi^{-1}(P)\right)\right>=\left<\varphi(\mathfrak{m})\right>\). On the other hand we know that \(\left<\varphi\left(\varphi^{-1}(P)\right)\right>\subset P\). Hence \(\left<\varphi(\mathfrak{m})\right>\subset P\). Since \(\left<\varphi(\mathfrak{m})\right>\) is assumed to be \(\mathfrak{n}\)-primary, we see that \(P=\mathfrak{n}\).
\end{proof}
\section{Main results}

\begin{theorem}\label{lift}
Let \((R,\mathfrak{m})\) be an equicharacteristic complete regular local ring and let \(A=R/\mathfrak{a}\) be a quotient ring with maximal ideal \(\mathfrak{m}_A\). Let \(\pi:R\longrightarrow A\) be the canonical surjection. Then any finite self map \(\varphi\) of \(A\) can be lifted to a finite self map \(\psi\) of \(R\) rendering a commuting diagram:

\[\begin{tikzpicture}
\matrix(m)[matrix of math nodes,
row sep=2.6em, column sep=2.8em,
text height=1.5ex, text depth=0.25ex]
{R&R\\
A&A\\};
\path[->,font=\scriptsize,>=angle 90]
(m-1-1) edge node[auto] {$\psi$} (m-1-2)
edge node[left] {$\pi$} (m-2-1)
(m-1-2) edge node[right] {$\pi$} (m-2-2)
(m-2-1) edge node[auto] {$\varphi$} (m-2-2);
\end{tikzpicture}\]
\end{theorem}
\begin{remark}\label{finite}
Let \((A,\mathfrak{m}_A)\) be a complete Noetherian local ring and let \(\varphi\) be a self map of \(A\). Then by~\cite[Theorem~8, p.~68]{ISCohen}, to say that \(\varphi\) is finite is the same thing as to say that \(\left<\varphi(\mathfrak{m}_A)\right>\) is \(\mathfrak{m}_A\)-primary and \([A/\mathfrak{m}_A:\varphi(A)/\varphi(\mathfrak{m}_A)]\) is a finite (algebraic) field extension.
\end{remark}
\begin{proof}(of Theorem~\ref{lift}) 
 Let \(K\) be an arbitrary coefficient field of \(R\). Then \(\varphi\left(\pi(K)\right)\) is a subfield of \(A\), and can be lifted to a subfield \(L\) of \(R\), by Lemma~\ref{lifting fields}, in such a way that \(\pi|_L:L\longrightarrow\varphi\left(\pi(K)\right)\) is an isomorphism. We will use \(L\) at the end of our proof to construct the self map \(\psi\) of \(R\). Let \(d=\dim A\) and let \(e\) be the embedding dimension of \(A\). By Lemma~\ref{SOP} we can choose a strong system of parameters \(\{x_1,\ldots,x_d\}\) of \(A\) which is part of a minimal set of generators \(\{x_1,\ldots,x_d,\ldots,x_e\}\) of \(\mathfrak{m}_A\). Choose elements \(X_1,\ldots,X_e\) in \(\mathfrak{m}\) in such a way that \(\pi\left(X_i\right)=x_i\) for each \(i\). We claim that since the images of \(x_1,\ldots,x_e\) in \(\mathfrak{m}_A/\mathfrak{m}_A^2\) are linearly independent over \(A/\mathfrak{m}_A\), the images \(\overline{X}_1,\ldots,\overline{X}_e\) of \(X_1,\ldots,X_e\) in \(\mathfrak{m}/\mathfrak{m}^2\) are also linearly independent over \(R/\mathfrak{m}\). If not, there will be a dependence relation \(\alpha_1\overline{X}_1+\ldots+\alpha_e\overline{X}_e=0\) with \(\alpha_i\in R/\mathfrak{m}\) not all zero. This means if we choose elements \(a_i\in R-\mathfrak{m}\) such that they map to \(\alpha_i\) in \(R/\mathfrak{m}\) for \(1\leq i\leq e\), then \(a_1X_1+\ldots+a_eX_e\in\mathfrak{m}^2\). If we apply \(\pi\) to this relation, we obtain \(\pi(a_1)x_1+\ldots+\pi(a_e)x_e\in\mathfrak{m}_A^2\). Thus, the image of this element in \(\mathfrak{m}_A/\mathfrak{m}_A^2\) will give a nontrivial dependence relation \(\pi(a_1)\overline{x}_1+\ldots+\pi(a_e)\overline{x}_e=0\), which contradicts the linear independence of \(\overline{x}_1,\ldots,\overline{x}_e\). This proves the above claim. Hence, we can extend \(\{\overline{X}_1,\ldots,\overline{X}_e\}\) to a basis \(\{\overline{X}_1,\ldots,\overline{X}_e,\ldots,\overline{X_n}\}\) of \(\mathfrak{m}/\mathfrak{m}^2\) over \(R/\mathfrak{m}\), where \(n=\dim R\). By Nakayama Lemma this means if for \(e+1\leq i\leq n\) we choose elements \(X_i\in\mathfrak{m}\) such that they map to \(\overline{X}_i\) in \(R/\mathfrak{m}\), then \(\{X_1,\ldots,X_n\}\) is a minimal set of generators of \(\mathfrak{m}\). Furthermore, it follows from the Cohen Structure Theorem that 
\[
R=K\llbracket X_1,\ldots,X_n\rrbracket. 
\]
Now consider the elements \(\varphi\left(\pi(X_i)\right)\) in \(A\) and for \(1\leq i\leq d\) choose elements \(f_i\in\mathfrak{m}\) such that \(\pi(f_i)=\varphi\left(\pi(X_i)\right)\). We claim that the ideal \(\left<f_1,\ldots,f_d\right>\) of \(R\) has height \(d\). Note that by Krull's Theorem we have \(\heit \left<f_1,\ldots,f_d\right>\leq d\). To see the inequality in the other direction, we first note that the ideal \(\mathfrak{b}:=\left<\varphi\left(\pi(X_1)\right),\ldots,\varphi\left(\pi(X_d)\right)\right>\) is \(\mathfrak{m}_A\)-primary. This follows by applying Lemma~\ref{nice} to a minimal prime ideal \(P\) of \(\mathfrak{b}\) and recalling that \(\varphi\) is finite and \(\{x_1,\ldots,x_d\}\) is a system of parameters of \(A\). Therefore the ideal \(\pi^{-1}(\mathfrak{b})=\left<f_1,\ldots,f_d\right>+\mathfrak{a}\) is \(\mathfrak{m}\)-primary in \(R\) and by Serre's Intersection Theorem~\cite[p.~136, Chap.~V, Theorem~1]{Serre}
\[\dim R/\mathfrak{a}+\dim R/\left<f_1,\ldots,f_d\right>\leq\dim R,\]
or \(d+\dim R/\left<f_1,\ldots,f_d\right>\leq n\). Since \(R\) is regular, we have \(\dim R/\left<f_1,\ldots,f_d\right>=n-\heit\left<f_1,\ldots,f_d\right>\) and we obtain \(\heit \left<f_1,\ldots,f_d\right>\geq d\). This and the inequality in the other direction establish our claim that \(\heit\left<f_1,\ldots,f_d\right>=d\).\par
Now we will choose elements \(f_{d+1},\ldots, f_n\) in \(\mathfrak{m}\) inductively, making sure at each step that 
 \(\dim R/\left<f_1,\ldots,f_t\right>= n-t\), for \(d+1\leq t\leq n\).  So suppose \(d\leq t <n\) and that \(f_1,\ldots f_t\) have been chosen so that \(\dim R/\left<f_1,\ldots,f_t\right> = n-t\).  To choose \(f_{t+1}\) we will apply the \emph{coset version} of the Prime Avoidance Lemma due to E.~Davis (\cite[Theorem~124, p.~90]{Kapl} or~\cite[Exercise~16.8, p.~133]{Matsumura2}), which can be stated as follows: \emph{let \(I\) be an ideal of a commutative ring \(R\) and \(x\in R\) be an element. Let \(\mathfrak{p}_1,\ldots,\mathfrak{p}_s\) be prime ideals of \(R\) none of which contain \(I\). Then \[x+I\not\subset\displaystyle{\bigcup_{i=1}^s}\ \mathfrak{p}_i.\] }
 We choose an element \(u\in\mathfrak{m}\) such that \(\pi(u)=\varphi\left(\pi(X_{t+1})\right)\).  If 
 \(\dim R/\left<f_1,\dots f_t, u\right>=n-t-1\), then we set  \(f_{t+1}:= u\).  Otherwise, we let \(\{\mathfrak{p}_{i}\}_{1\leq i\leq s}\) be the set of all minimal associated prime ideals of \(R/\left<f_1,\ldots,f_{t}\right>\) that satisfy \(\dim R/\mathfrak{p}_i=\dim R/\left<f_1,\ldots,f_t\right>\).  None of these \(\mathfrak{p}_i\)'s contain \(\mathfrak{a}\), because \(\left<f_1,\ldots,f_t\right>+\mathfrak{a}\) is an \(\mathfrak{m}\)-primary ideal in \(R\). Therefore by the coset version of the Prime Avoidance Lemma there exists an element \(a\in\mathfrak{a}\) such that
\[u+a\not\in\bigcup_{i=1}^s\ \mathfrak{p}_i.\]
Setting \(f_{t+1} := u+ a\) we obtain \(\dim R/\left<f_1,\ldots,f_{t+1}\right> = n-t-1\) and \(\pi(f_{t+1})=\varphi\left(\pi(X_{t+1})\right)\). After choosing \(\{f_1,\ldots,f_n\}\) as described, we define a self map \(\psi\) of \(R=K\llbracket X_1,\ldots,X_n\rrbracket\) as follows. For each \(1\leq i\leq n\), we define \(\psi(X_i)\) to be \(f_i\) and for every element \(\alpha\) of \(K\) we define \(\psi(\alpha)\) to be \(\left(\pi|_L\right)^{-1}\left(\varphi\left(\pi(\alpha)\right)\right)\). Then \(\psi(\mathfrak{m})=\left<f_1,\cdots,f_n\right>\) is \(\mathfrak{m}\)-primary by construction of the \(f_i\)'s, and since \(\varphi\) is a finite self map of \(A\), we can see by Remark~\ref{finite}, that \(\psi\) is also a finite map.
\end{proof}
\begin{corollary}\label{shift}
Let \((A,\mathfrak{m}_A)\) be an equicharacteristic complete Noetherian local ring and let \(K\) be an arbitrary coefficient field of \(A\). Then there exists a surjective homomorphism \(\pi\) from some power series ring \(R\) over \(K\) onto \(A\), with the property that any finite self map \(\varphi\) of \(A\) can be lifted to a finite self map \(\psi\) of \(R\) such that \(\varphi\circ\pi=\pi\circ\psi\).
\end{corollary}
\begin{proof}
 Let \(d=\dim A\) and let \(e\) be the embedding dimension of \(A\). By Lemma~\ref{SOP} we can choose a strong system of parameters \(\{x_1,\ldots,x_d\}\) of \(A\) which is part of a minimal set of generators \(\{x_1,\ldots,x_d,\ldots,x_e\}\) of \(\mathfrak{m}_A\). Set \(R:=K\llbracket X_1,\ldots,X_e\rrbracket\), where \(X_1,\ldots,X_e\) are indeterminates. Then, by the Cohen Structure Theorem there is a surjective homomorphism \(\pi:R\longrightarrow A\), such that \(\pi\left(X_i\right)=x_i\) for each \(1\leq i\leq e\). Since \(R\) is a complete regular local ring and \(A\) is a homomorphic image of \(R\), the result follows from Theorem~\ref{lift}.
\end{proof}

\begin{remark}
Let \(X\) be a projective variety over a field \(K\) with a self map \(\varphi\), and let \(\mathcal{L}\) be an ample line bundle on \(X\) such that \(\varphi^*(\mathcal{L})\cong\mathcal{L}^{\otimes d}\) for some \(d\geq1\). Then some appropriate tensor power \(\mathcal{L}^{\otimes n}\) of \(\mathcal{L}\) is very ample and can be used to embed \(X\) in some projective space \(\mathbb{P}^N_K\), realizing \(X\) as \(\proj\left(K[X_1,\cdots,X_N]/\mathfrak{a}\right)\) for some graded ideal \(\mathfrak{a}\). Let 
\[\pi: K[X_1,\ldots,X_N]\longrightarrow K[X_1,\cdots,X_N]/\mathfrak{a}\] be the canonical surjection, and let \(\mathfrak{m}=\left<X_1,\cdots,X_N\right>\) and \(\mathfrak{m}_X=\left<\pi(X_1),\cdots,\pi(X_N)\right>\) be the corresponding irrelevant maximal ideals, respectively. Then \(\varphi\) will induce a graded \(K\)-self map of \(K[X_1,\cdots,X_N]/\mathfrak{a}\), which we will denote by \(\varphi\) again, with the property that \(\varphi(\mathfrak{m}_X)\) is \(\mathfrak{m}_X\)-primary. Then the proof of Theorem~\ref{lift} can be re-written in this setting, keeping careful track of grading, to lift \(\varphi\) to a graded \(K\)-self map \(\psi\) of \(K[X_1,\ldots,X_N]\) in such a way that \(\psi(\mathfrak{m})\) is \(\mathfrak{m}\)-primary. This shows that the assumption in~\cite[Corollary~2.2]{NF}, that \(K\) is infinite can be avoided.
\end{remark}

\bibliographystyle{amsplain}

\end{document}